\documentclass[12pt]{article}
\usepackage{hyperref}

\newcommand{\RNum}[1]{\uppercase\expandafter{\romannumeral #1\relax}}
\usepackage{amsmath}
\usepackage{hyperref}
\usepackage{amssymb}
\usepackage{latexsym}
\usepackage{epsf}
\usepackage{supertabular}
\usepackage{rotating}
\usepackage{multicol}
\usepackage{multirow}
\usepackage{epsfig}
\usepackage{arydshln}
\usepackage{fancyhdr}       
\usepackage{indentfirst}
\newif\ifger
\gerfalse

\newtheorem{theorem}{Theorem}[section]

\newtheorem{lemma}[theorem]{Lemma}

\newtheorem{corollary}[theorem]{Corollary}

\newtheorem{proposition}[theorem]{Proposition}

\title{Reduction for flag-transitive symmetric $2$-designs with $\gcd(k,\lambda)=p$ }

\author{
	Ziwei Lu$^{1}$\footnote{Corresponding author. lzw19980928@163.com}, Yu Sun$^{2}$, Shenglin Zhou$^{2}$ \\[4pt]
	\small
	\parbox{0.9\textwidth}{\centering
		$^{1}$ \textit{School of Artificial Intelligence, Hubei Polytechnic University,} \\
		\hspace*{1.2em} \textit{Huangshi 435003, China}\\[4pt]
		$^{2}$ \textit{School of Mathematics, South China University of Technology,} \\
		\hspace*{1.2em} \textit{Guangzhou 510641, China} \\
	}
}

\begin{document}
	\baselineskip=19pt
	\date{}
	\maketitle

\begin{abstract}
		We study symmetric $2$-$(v,k,\lambda)$ designs admitting a flag-transitive point-primitive subgroup $G$ of automorphisms, where $\gcd(k,\lambda)=p$ and $p$ is a prime. 
		In this paper, we show that $G$ is of affine or almost simple type.
		
		
	\end{abstract}

	\section{Introduction}
	
	\indent A symmetric $2$-$(v, k, \lambda)$ design $\mathcal{D}=(\mathcal{P}, \mathcal{B})$ is a set $\mathcal{P}$ of $v$ points, together with a block collection $\mathcal{B}$ consisting of some $k$-subsets of $\mathcal{P}$, such that the number $v$ of points is equal to the number $b$ of blocks and any two distinct points of $\mathcal{P}$ are contained in $\lambda$ blocks. Since all the blocks have the same size $k$, then the number of blocks containing a given point is a constant, denoted by $r$. In a symmetric $2$-design, it is easy to see that $r=k$. A \emph{flag} of $\mathcal{D}$ is a point-block pair $(\alpha,B)$ such that $\alpha\in B.$ If $G\le Aut(\mathcal{D})$ acts transitively on the set of flags, then $G$ is called \emph{flag-transitive}. And the group $G$ is called \emph{point-primitive}, if it acts primitively on the set of points.  Otherwise, $G$ is called \emph{point-imprimitive}.\\
	\indent Dating back to 1968, Dembowski \cite{Dembowski} first pointed out that primitivity can be derived from flag-transitivity under the condition $\gcd(r,\lambda)=1$. Subsequently, a seminal result of Zieschang \cite{1988} showed that, for any $2$-designs $\mathcal{D}$ with $\gcd(r,\lambda)=1$, if $G\le Aut(\mathcal{D})$ is flag-transitive and $T$ is a non-abelian minimal normal subgroup of $G$, then $T$ is simple and $C_{G}(T)=1$. It follows that $G$ is an affine or almost simple group. Namely, the socle of $G$ is an elementary abelian group or a
	non-abelian simple group. This suggests us to consider the flag-transitive $2$-designs with $\gcd(r,\lambda)\neq 1$, especially for the case that $\gcd(r,\lambda)$ is a prime. In 2021, Zhao and Zhou \cite{zhao} studied the flag-transitive 2-designs $\mathcal{D}$ with $\gcd(r,\lambda)=2$ and $\lambda\neq 2$, they proved that the automorphism group $G$ of $\mathcal{D}$ is point-primitive and of affine or almost simple type. Note that, for the case $\lambda=2$, there exists two flag-transitive and point-imprimitive symmetric $2$-$(16,6,2)$ designs (see  \cite{regueiro}).\\
	\indent The purpose of this paper is to investigate symmetric 2-$(v,k,\lambda)$ designs with $\gcd(k,\lambda)=p$, where $p$ is a prime, admitting a flag-transitive point-primitive automorphism group. The outline of the proof is as follows. First, using the relationship between the parameters, we show in Propositions \ref{prop3.11} and \ref{prop3.1} that the automorphism group $G$ of $2$-designs cannot be of simple diagonal and twisted wreath type. Second, we focus on the case that $G$ is of product type.
	 In fact, there are also other studies of symmetric 2-designs relative to the automorphsim group $G$ with product action. For example, O'Reilly-Regueiro and Rodr{\'i}guez-Fitta studied in \cite{2023} that $\mathcal{D}$ is a symmetric 2-design with  a flag-transitive point-primitive automorphism group $G\le H\wr S_2$ of product action type. We will use the results of this article to discuss the case when the automorphism group $G$ of $\mathcal{D}$ has a primitive product action.\\
	 \indent Based on the O'Nan-Scott Theorem for finite primitive groups, we obtained the following reduction theorem of flag-transitive point-primitive symmetric 2-designs with $\gcd(k,\lambda)=p$.
	\begin{theorem}\label{TH1}
		Let $\mathcal{D}$ be a non-trivial symmetric $2$-$(v, k, \lambda)$ design with $\gcd(k,\lambda)=p$, where $p$ is a prime. If $G\le Aut(\mathcal{D})$ is flag-transitive and point-primitive, then the group $G$ is either of affine or almost simple type.
	\end{theorem}
	
	\section{Preliminaries}
	We start with some facts about 2-designs used in this article. Recall that $v,b,r,k,\lambda$ are called the parameters of designs.
	\begin{lemma}\label{l1} {\rm(\cite{Dav1}, \cite{Osifodunrin})}
		Let $\mathcal{D}=(\mathcal{P},\mathcal{B})$ be a symmetric $2$-$(v, k, \lambda)$ design, $\alpha\in \mathcal{P}$  and $G$ be a flag-transitive automorphism group of $\mathcal{D}$, then the following hold:
		\begin{enumerate}
			\item[\rm (1)] $\lambda (v-1)=k(k-1)$. In particular, $\lambda v<k^2$;
			\item[\rm (2)] $G$ is point-transitive, and $G_{\alpha}$ acts transitively on the set of blocks through $\alpha$. Furthermore, $k\mid |G_{\alpha}|$;
			\item[\rm (3)] If $d$ is any non-trivial subdegree of $G$, then $k \mid \lambda d$. 
		\end{enumerate}
	\end{lemma}
\begin{lemma}\label{ll2.2}{\rm(\cite[Proposition 1.1]{ca})}
	Let $\mathcal{D}=(\mathcal{P},\mathcal{B})$ be a symmetric $2$-$(v, k,\lambda)$ design, admitting a flag-transitive  automorphism group $G$. If $K$ is a permutation group with $G\le K\le Sym(\mathcal{P})$, then $(\mathcal{P},\mathcal{B}^K)$ is also a $2$-$(v,k,\lambda)$ design, admitting a flag-transitive automorphism group $K$.
\end{lemma}

\indent Next, we list some conclusions about the primitive group of product type. Let $H$ be a primitive permutation group on $\Gamma$ of size $w$, where $H$ is of almost simple type or simple diagonal type. Assume that $G$ has a product action on $\mathcal{P}=\Gamma^{\ell}$ $(\ell\ge 2)$, then 
$T_1\times T_2\times \cdots \times T_{\ell}\unlhd G\le H\wr S_{\ell},$ where all $T_i$ are isomorphic to each other and $Soc(H)=T_i$ for $i\in \left\{1,\dots, \ell\right\}$.
\begin{lemma}\label{lemma3.2}{\rm(\cite[Lemma 2.3]{lu})}
	Assume that $\mathcal{D}$ is a symmetric $2$-$(v,k,\lambda)$ design and $G\le Aut(\mathcal{D})$ is flag-transitive. Moreover, if $G$ is of product type, then
	\begin{enumerate}
		\item[\rm (1)] $ k \mid \lambda \ell(w-1)$. In particular, $k^*\mid \ell(w-1),$ where $k^*=\frac{k}{\gcd(k,\lambda)}$.
		\item[\rm (2)] If $\gcd(k,\lambda)=p^e$ where $e$ is a positive integer, then
		\begin{equation*}
			\ell^2 p^e=(w^{\ell -2}+2w^{\ell -3}+\cdots +(\ell-2)w+\ell-1)\lambda^*x^2+\dfrac{\ell x+\ell \lambda^*x^2}{w-1},
		\end{equation*}
		where $x$ is a positive integer and $\lambda^*=\frac{\lambda}{\gcd(k,\lambda)}$.
	\end{enumerate}
\end{lemma}

\indent Now assume that $\ell=2$. Let $\alpha=(\alpha_{1},\alpha_{2})\in B$ and $\mathcal{G}_{\alpha}=(\alpha^{T_1}\cup \alpha^{T_2})\setminus\{\alpha\}$. 
Counting $\Sigma_1=\{(\beta,B)|\beta\in B\cap \mathcal{G}_{\alpha}, \alpha\in B\}$ in two ways, we get $\lambda|\mathcal{G}_{\alpha}|=k|B\cap \mathcal{G}_{\alpha}|$. Putting $x_1=|B\cap \mathcal{G}_{\alpha}|$, then
 $$x_1=|B\cap \mathcal{G}_{\alpha}|=\dfrac{2\lambda(w-1)}{k},$$
 which corresponds to Lemma \ref{lemma3.2}(1) when $\ell=2$.\\
\indent Let $M=N_G\left(T_1\right) \cap N_G\left(T_2\right).$ Obviously, $G$ acts transitively by conjugation on the set $\{T_1, T_2\}$. Therefore, $|G:M|=2$ and $|G_{\alpha B}:M_{\alpha B}|\le 2$. Since $M\unlhd G$ and
\begin{equation*}
	vk=|G:G_{\alpha B}|=\dfrac{|G:M||M:M_B||M_B:M_{\alpha B}|}{|G_{\alpha B}:M_{\alpha B}|},
\end{equation*}
then $(|B^M|,|\alpha^{M_B}|)=(v,k), (v,\frac{k}{2})$ or $(\frac{v}{2},k)$.\\
\indent If $|\alpha^{M_B}|=\frac{k}{2}$, then $|G_B:M_B|=|G:M|=2$. There exists $g\in G_B\setminus M_{B}$ such that $T_1^g=T_2$ and $T_2^g=T_1$. Let
$$
\begin{array}{lll}
	B_{1j}&=&\{\alpha_{2}|(\alpha_{1},\alpha_{2})\in \alpha^{g^{j}M_{B}}\}, \\
B_{2j}&=&\{\alpha_{1}|(\alpha_{1},\alpha_{2})\in \alpha^{g^{j}M_{B}}\},
\end{array}
 $$ 
 for $j\in \{0,1\}$.
\begin{lemma}\label{lemma2.4}{\rm(\cite[Lemmas 4.6, 4.8]{100})}
	Suppose that $\mathcal{D}$ is a symmetric $2$-$(v,k,\lambda)$ design, and $G\le H\wr S_2$ is a flag-transitive point-primitive product action automorphism group. The following hold:
	\begin{enumerate}
		\item[\rm (1)] If $x_1$ is odd, then $(|B^M|,|\alpha^{M_B}|)\neq (v,k)$.
		\item[\rm (2)] If $|\alpha^{M_B}|=\frac{k}{2}$ and 
		\begin{enumerate}
		\item[\rm (i)] $B_{20}\cap B_{21}=\emptyset$, then $|\alpha^{M_B}|=|B\cap \alpha^{T_2}||B_{20}|$;
		\item[\rm (ii)] $B_{20}\cap B_{21}\neq \emptyset$, then $x_1$ is even.
		\end{enumerate}
\end{enumerate}	
\end{lemma}

\indent Note that $|\alpha^{M_B}|=\frac{k}{2}$ or $k$, the following lemma give the composition of $x_1=|B\cap \mathcal{G}_{\alpha}|$.
\begin{lemma}\label{l2} {\rm (\cite[Lemma 2.9]{on})} 
 		Assume that $\mathcal{D}=(\mathcal{P},\mathcal{B})$ and $G$ satisfy the conditions of Lemma {\rm\ref{lemma2.4}}.
 	For $\alpha=(\alpha_1,\alpha_2) \in B$, let $\left|\alpha^{T_1} \cap \alpha^{M_B}\right|=c_1,\left|\alpha^{T_2} \cap \alpha^{M_B}\right|=c_2.$ Then one of the following holds:
	\begin{enumerate}
		\item[\rm (1)] $\left|\alpha^{M_B}\right|=\frac{k}{2}$, and $x_1=c_1+c_2-2$ or $2\left(c_1+c_2\right)-2$, where $c_i \mid \frac{k}{2}$, for $i=1,2$;
		\item[\rm(2)]$\left|\alpha^{M_B}\right|=k$, and $x_1=c_1+c_2-2$, where $c_i \mid k$, for $i=1,2$.
	\end{enumerate}
\end{lemma}

\indent Recall that $G\le H\wr S_2$, where $H$ is of almost simple or simple diagonal type. In particular, if $H$ is of almost simple type, we have the following result. 
 \begin{lemma}\label{lemma3.3}{\rm(\cite[Lemmas 3.3, 3.7]{2020})}
	Assume that $\mathcal{D}=(\mathcal{P},\mathcal{B})$ and $G$ satisfy the conditions of Lemma {\rm\ref{lemma2.4}}. Let $T_1\times T_2\unlhd G\le H\wr S_2$, where $H$ is an almost simple type primitive group on $\Gamma$ of size $w$ with simple socle $T_i\cong T$ $(i=1,2)$ and $T$ is a non-abelian simple group. Let $\alpha\in \mathcal{P}$ and $B\in \mathcal{B}$ such that $\alpha\in B$, then
	\begin{enumerate}
		\item[\rm (1)] $k\mid 2|\alpha^{\left(T_{1}\right)_{B}}||\alpha^{\left(T_{2}\right)_{B}}||Out(T)|^2w^4$;
		\item[\rm (2)] $k=\dfrac{(|\alpha^{T_{1}}\cap B|+|\alpha^{T_{2}}\cap B|-2)(w+1)}{2}+1$.
	\end{enumerate}
\end{lemma}

\section{Proof of Theorem \ref{TH1}}
 \numberwithin{equation}{section}
\indent In this section, we always assume that  $\mathcal{D}=(\mathcal{P}, \mathcal{B})$ is a symmetric $2$-$(v, k, \lambda)$ design admitting a flag-transitive point-primitive automorphism group $G.$ Recall that $\gcd(k,\lambda)=p$. Let $k=k^*p$, $\lambda=\lambda^*p$, where $$k^*=\frac{k}{\gcd(k,\lambda)}~\text{and}~ \lambda^*=\frac{\lambda}{\gcd(k,\lambda)}.$$ 
By \cite[Theorem 1.1]{zhao}, we know that Theorem \ref{TH1} holds for $p=2$. So we suppose that $p$ is an odd prime in the following. Moreover, if $\lambda^*=1$, then $\lambda=p$, which is a prime number. By the result of \cite{2020}, Theorem \ref{TH1} holds, so we further assume that $\lambda^*\ge 2$. It yields that $k^*\ge 3.$\\
\indent Based on O'Nan-Scott Theorem \cite{li}, we know that the primitive permutation groups can be classified into five types: (i) Simple diagonal type;  (ii) Twisted wreath product type; (iii) Product type; (iv) Affine type;  (v) Almost simple type. We first show that $G$ cannot be of twisted wreath and simple diagonal type.
\begin{proposition}\label{prop3.11}{\rm({\cite[Proposition 3.3]{chen}})}
	If $G$ is a flag-transitive point-primitive automorphism group of a $2$-design, then $G$ cannot be primitive of twisted wreath product type.
\end{proposition}
\begin{proposition}\label{prop3.1}
	If $\mathcal{D}=(\mathcal{P}, \mathcal{B})$ is a symmetric $2$-$(v, k, \lambda)$ design admitting a flag-transitive point-primitive automorphism group $G$, where  $\gcd(k,\lambda)=p$ is a prime, then $G$ is not of simple diagonal type.
\end{proposition}
	{\bfseries Proof}. Assume to the contrary that $G$ is of simple diagonal type. Let $N \cong T^{m}$ $(m\ge 2)$ be the socle of $G$, where $T$ is a non-abelian simple group. For a point $\alpha \in B$, by the flag-transitivity of $G$ and $N_{\alpha} \lhd G_{\alpha}$, we have $\left|N_{\alpha}: N_{\alpha B}\right| \mid k$. Using Lemma \ref{l1}(1), we obtain that 
\begin{equation*}
	|N_{\alpha}: N_{\alpha B}| \mid k\Big( \frac{k-1}{\lambda^*}\Big)=p(v-1).
\end{equation*}
Noting $v=|T|^{m-1},$ which implies that $\gcd(\left|N_{\alpha}: N_{\alpha B}\right|, v-1)=1$, and so $\left|N_{\alpha}: N_{\alpha B}\right| \mid p.$ Therefore, $\left|N_{\alpha}: N_{\alpha B}\right|=1$ or $p.$\\
\indent If $\left|N_{\alpha}: N_{\alpha B}\right|=1$, by \cite[Proposition 2]{1988}, then $N$ is a Frobenius group, contrary to the fact that $N$ is the direct product of some non-abelian simple groups. So, $\left|N_{\alpha}: N_{\alpha B}\right|=p$. It follows that $p\mid |N|$, and so $p\mid v$. The basic equation gives that $\lambda^*(v-1)=k^*(k^*p-1)$. So $p\mid (k^*-\lambda^*)$, and hence $p<p+\lambda^*\le k^*$. By
Lemma \ref{l1}(1), $\lambda v<k^2$, which means that $v<\frac{k^{*3}}{\lambda^*}\le \frac{1}{2}k^{*3}$ for $\lambda^*\ge 2$. Moreover, $k^*\le m|T|$ by Lemma \ref{l1}(3). Hence, 
\begin{equation}\label{eq3.1}
	|T|^{m-1}=v<\frac{1}{2}k^{*3}\le \frac{1}{2}m^3|T|^3.
\end{equation}
As $T$ is a non-abelian simple group, we have $60^{m-4}\le |T|^{m-4}<\frac{1}{2}m^3$, so $2\le m\le 5.$\\
\indent Since $G$ is of simple diagonal type, then $G_{\alpha}\lesssim Aut(T) \times S_{m}$. This implies, by Lemma \ref{l1}(2), that $k\mid m!|T||Out(T)|$. Note that $k^*\mid (v-1)$, then $\gcd(k^*,|T|)=1$, and hence $k^*\mid m!|Out(T)|$. Thus
\begin{equation}\label{eq3.2}
	|T|^{m-1}=v<\frac{1}{2}k^{*3}\le \frac{1}{2}m!^3|Out(T)|^3.
\end{equation}
If $m=5$, by (\ref{eq3.2}) $|T|\le 30|Out(T)|^{3/4}$, and we deduce that $|T|<400|Out(T)|^{2}$. These simple groups are determined in \cite[Lemma 2.2]{lu}. That is, $T\cong L_{3}(4)$ or $L_{2}(q)$ for $q\in \{5,7,$ $8,9,11,13,16,27\}$. 
But there is no group that satisfying the inequality $|T|<30|Out(T)|^{3/4}$. Again using (\ref{eq3.2}), we have
$|T|\le 19|Out(T)|$ when $m=4$, and $|T|\le 10|Out(T)|^{3/2}$ when $m=3$. In the same way, no simple group satisfying these inequalities. If $m=2$, then $|T|<4|Out(T)|^3$ by (\ref{eq3.2}). It also follows from \cite[Lemma 2.2]{lu} that $T$ is isomorphism to $A_{5}$ or $A_{6}$. In both cases, however, the inequality $|T|<4|Out(T)|^3$ is violated.\\
\indent Hence, $G$ is not of simple diagonal type.\vspace{3mm}$\hfill\square$\\
\indent Now assume that $G$ is a primitive group of product action type. Let $H$ be of almost simple or simple diagonal type on a set $\Gamma$ of size $w$, and $G \leq H\wr S_{\ell}$ be a primitive group acting flag-transitively on $\mathcal{D}$, with a product action on $\mathcal{P}=\Gamma^{\ell}$ $(\ell\ge 2)$. We first will reduce to $\ell=2.$
	\begin{lemma}\label{lp}
		Assume that $\mathcal{D}$ is a symmetric $2$-$(v,k,\lambda)$ design, where the parameters satisfying $\gcd(k,\lambda)=p$ is a prime. Let $G\le Aut(\mathcal{D})$ be a flag-transitive point-primitive group with product action. Then $p\le w$ and $\ell=2$.
	\end{lemma}
	{\bfseries Proof}. Since $G$ is of product action type, then $G \leq H^{\ell} \rtimes S_{\ell}\leq (S_w)^{\ell}\rtimes S_{\ell}$. It follows that $|G|\mid w!^{\ell}\ell!$. The flag-transitivity of $G$ gives $k\mid w!^{\ell}\ell!$, and hence $p\mid w!^{\ell}\ell!$. Suppose that $p$ is a divisor of $\ell!$, then $p\le \ell$. Using Lemma \ref{l1}(1), the equation can be written as $v=\frac{k^*(k^*p-1)}{\lambda^*}+1$, and so $v<\frac{1}{2}k^{*2}p$ by $\lambda^*\ge 2$. Moreover, by Lemma \ref{lemma3.2}(1), $k^*\mid \ell(w-1)$. Thus, 
	\begin{equation}\label{eq3.3}
		w^\ell=v<\frac{1}{2}k^{*2}p\le \frac{1}{2}p\ell^2(w-1)^2\le \frac{1}{2}\ell^3(w-1)^2<\frac{1}{2}\ell^3w^2.
	\end{equation}
It yields that $w^{\ell-2}<\frac{1}{2}\ell^3$, and thus $\ell\le 4.$ From $p\le \ell \le 4$ and $p$ is an odd prime, we get $p=3$. Combining it with (\ref{eq3.3}), we have $\ell=3$ and 
$5\le w\le11$. Further, the value of $k^*$ is determined by $k^*\mid \ell(w-1)$. But we find that $\lambda^*=\frac{k^*(k^*p-1)}{w^\ell-1}=\frac{k^*(3k^*-1)}{w^3-1}$ is not an integer for each pair $(w,k^*)$, a contradiction.\\
\indent Hence, $p\mid w!$, and so $p\le w.$ Recall that $w^{\ell}<\frac{1}{2}k^{*2}p$. Then $w^{\ell}< \frac{1}{2}w\ell^2(w-1)^2<\frac{1}{2}\ell^2w^3$, and so $w^{\ell-3}<\frac{1}{2}\ell^2$. Since $w\ge 5$, we get $\ell=2, 3$ or $(\ell,w)=(4,5)$, $(4,6)$, $(4,7),$ respectively. The last two cases are excluded as they do not satisfy the inequality $w^{\ell}< \frac{1}{2}w\ell^2(w-1)^2$. If $(\ell,w)=(4,5)$ then $v=625$, and since $k^*\mid \ell(w-1)$ we have $k^*=4, 8$ or $ 16$. If $k^*=4$, Lemma \ref{l1}(1) yields $156\lambda^*=4p-1$, which is impossible. Similarly, $k^*\neq 8$. If $k^*=16$, again using Lemma \ref{l1}(1), we have $39\lambda^*=16p-1$. Note that $2\le \lambda^*<k^*$ and $\lambda^*$ is odd, for each such $\lambda^*$, the previous equation has no possible solution for $p$, a contradiction. Therefore, $\ell=2$ or $3$.\\
\indent Suppose next that $\ell=3$. Let $3(w-1)=k^*x$, where $x$ is a positive integer. It follows from Lemma \ref{lemma3.2}(2) that
\begin{equation}\label{eq3.5}
	9p=(w+2)\lambda^*x^2+\dfrac{3x+3\lambda^*x^2}{w-1},
\end{equation}
and so $(w+2)\lambda^*x^2<9p$. Combining it with $p\le w$, we have $\lambda^*x^2<9$, and hence $x\le 2$. By (\ref{eq3.5}), $(w-1)\mid (3x+3\lambda^*x^2)$, this implies that $w-1\le 3x+3\lambda^*x^2<33$. Thus $5\le w\le 33.$
For these values of $w,$ $\lambda^*$ and $x$, we observe that there exist exactly two solutions satisfying Equation (\ref{eq3.5}). That is, $(\lambda^*,x,w,p)=(6,1,8,7)$ or $(8,1,10,11)$. The first can be ruled out by $k^*=3(w-1)$ and $\gcd(k^*,\lambda^*)=1$. It is easy to see that $(v,k,\lambda)=(10^3,297,88)$ for the latter case. Note that $G$ is flag-transitive, then $vk\mid |G|$, and so $G=A_{1000}$ or $S_{1000}$ by the list of primitive groups of degree $10^3$ (see \cite{magma}). This contradicts the fact that $G$ is of product action. Hence, $\ell=2$.
$\hfill\square$\vspace{3mm}\\
\indent If $\ell=2$, from a geometry perspective, then there are two Cartesian lines through any point in the design. In other words, putting
 $\alpha=(\alpha_1,\alpha_2)\in \mathcal{P}$, the two Cartesian lines through $\alpha$ are
$\mathcal{G}_{\alpha,1}=\{(\gamma,\alpha_2)|\gamma\in \Gamma\}=\Gamma\times \{\alpha_{2}\}$ and $\mathcal{G}_{\alpha,2}=\{(\alpha_1,\gamma)|\gamma\in \Gamma\}=\{\alpha_{1}\}\times \Gamma$. Moreover, $\mathcal{G}_{\alpha,1}=\alpha^{T_1}$ and $\mathcal{G}_{\alpha,2}=\alpha^{T_2}$.\\
\indent The next lemma first appeared in \cite[Theorems 3.7-3.8]{2023}. However, the proof contains some errors, especially in \cite[Theorem 3.7]{2023}, so for completeness, we give a new proof under the condition that $\gcd(k,\lambda)$ is an odd prime based on the methods in \cite{2023,100}.
\begin{lemma}\label{prop2.6}
	Let $\mathcal{D}$ be a symmetric $2$-$(v,k,\lambda )$ design with $v=w^2$ and $\gcd(k,\lambda)$ odd prime, admitting a flag-transitive, point-primitive automorphism group $G$ with product action. And let $x_1=\frac{2\lambda(w-1)}{k}$. Then exactly one of the following cases holds.
	\begin{enumerate}
		\item[\rm (1)]There exists a flag $(\alpha,B)$ such that $B$ intersects only one Cartesian line through $\alpha$, and $x_1+1$ divides $k$;
		\item[\rm (2)]Every block through any fixed point $\alpha$ intersects each of the Cartesian lines through $\alpha$, and $\frac{x_1}{2}+1$ divides $k$.
	\end{enumerate}
\end{lemma}
{\bfseries Proof}. (1) Let $(\alpha,B)$ be the flag such that $B$ intersects a Cartesian line through $\alpha=(\alpha_{1},\alpha_{2})$. Without loss of generality, suppose that $B$ intersects the second Cartesian line through $\alpha$. That is, $(B\cap \mathcal{G}_{\alpha,1})\setminus \{\alpha\}=\emptyset$ and $(B\cap \mathcal{G}_{\alpha,2})\setminus \{\alpha\}\neq\emptyset$. Recall that $\mathcal{G}_{\alpha}=(\alpha^{T_1}\cup \alpha^{T_2})\setminus\{\alpha\}$, then $\mathcal{G}_{\alpha}=(\mathcal{G}_{\alpha,1}\cup \mathcal{G}_{\alpha,2})\setminus\{\alpha\}$.
Obviously, $|B\cap \mathcal{G}_{\alpha,2}|=|B\cap \mathcal{G}_{\alpha}|+1=x_1+1$.\\
\indent We assert that for any element $\beta=(\beta_1,\beta_2)$ of the block $B$, $B$ intersects only one Cartesian line through $\beta$. In fact, since $G$ is flag-transitive, then there exists $g=(g_1,g_2;\sigma) \in G$ such that $(\alpha,B)^g=(\beta,B)$. In other words, $\alpha^g=\beta$ and $B^g=B$. It follows from $G \le H\wr S_2$ that $g=(g_1,g_2;(1))$ or $(g_1,g_2;(12))$. If $g=(g_1,g_2;(1))$, then $\alpha_{1}^{g_1}=\beta_{1}$ and $\alpha_{2}^{g_2}=\beta_{2}$. It implies that
\begin{equation*}
	\begin{split}
		\mathcal{G}_{\alpha,1}^g=(\Gamma\times \{\alpha_{2}\})^{(g_1,g_2;(1))}=\Gamma\times \{\alpha_{2}^{g_2}\}=\Gamma\times \{\beta_{2}\}=\mathcal{G}_{\beta,1},\\
	\mathcal{G}_{\alpha,2}^g=(\{\alpha_{1}\}\times \Gamma)^{(g_1,g_2;(1))}= \{\alpha_{1}^{g_1}\}\times \Gamma= \{\beta_{1}\}\times \Gamma=\mathcal{G}_{\beta,2}.
	\end{split}
\end{equation*}
So $(B\cap \mathcal{G}_{\alpha,2})^g=B\cap  \mathcal{G}_{\beta,2}$, which means that $|B\cap  \mathcal{G}_{\beta,2}|=|B\cap  \mathcal{G}_{\alpha,2}|=x_1+1$. Note that $\mathcal{G}_{\alpha}^g={((\mathcal{G}_{\alpha,1}\cup \mathcal{G}_{\alpha,1})\setminus\{\alpha\})}^g=(\mathcal{G}_{\beta,1}\cup \mathcal{G}_{\beta,1})\setminus\{\beta\}=\mathcal{G}_{\beta}$. It yields that $|B\cap \mathcal{G}_{\beta}|=|B\cap \mathcal{G}_{\alpha}|=x_1$, and so $(B\cap \mathcal{G}_{\beta,1})\setminus \{\beta\}=\emptyset$. In this way $B$ intersects only the second Cartesian line through $\beta$. If $g=(g_1,g_2;(12))$, then $\alpha_{1}^{g_1}=\beta_{2}$ and $\alpha_{2}^{g_2}=\beta_{1}$. Moreover, we get 
\begin{equation*}
	\begin{split}
		\mathcal{G}_{\alpha,1}^g=(\Gamma\times \{\alpha_{2}\})^{(g_1,g_2;(12))}=\{\alpha_{2}^{g_2}\}\times \Gamma=\{\beta_1\}\times \Gamma=\mathcal{G}_{\beta,2},\\
		\mathcal{G}_{\alpha,2}^g=(\{\alpha_{1}\}\times \Gamma)^{(g_1,g_2;(12))}=\Gamma\times  \{\alpha_{1}^{g_1}\}=\Gamma\times \{\beta_{2}\}=\mathcal{G}_{\beta,1}.
	\end{split}
\end{equation*}
It implies that $|B\cap  \mathcal{G}_{\beta,1}|=|B\cap  \mathcal{G}_{\alpha,2}|=x_1+1$. Combining it with $|B\cap \mathcal{G}_{\beta}|=|B\cap \mathcal{G}_{\alpha}|=x_1$, then $(B\cap \mathcal{G}_{\beta,2})\setminus \{\beta\}=\emptyset$, and so $B$ intersects only the first Cartesian line through $\beta$. Thus the assertion is proved.\\
\indent Let $B_0=B\cap \mathcal{G}_{\alpha,2}$, the size of this set is $x_1+1$. Now choose an element $\gamma_1\in B\setminus B_0$, by above assertion, we know that $B$ intersects only one Cartesian line through $\gamma_1$. Without loss of generality, we assume that $(B\cap \mathcal{G}_{\gamma_1,1})\setminus \{\gamma_1\}\neq \emptyset$, and let $B_1=B\cap \mathcal{G}_{\gamma_1,1}$. Clearly, $|B_1|=x_1+1$. In the same way as before, we take $\gamma_2\in B\setminus(B_0\cup B_1)$, and then $B$ intersects only one Cartesian line through $\gamma_2$. Putting $B_2=B\cap \mathcal{G}_{\gamma_2,2}.$ The process is continued in this way until no more points can be taken in $B$, thus we obtain a collection of points $\gamma_0=\alpha$, $\gamma_1, \dots, \gamma_i$, where $\gamma_j\in B\setminus (B_0\cup\dots \cup B_{j-1})$ for $0\le j\le i$ and $B_j$ as the intersection of $B$ with the only Cartesian line through $\gamma_j$ that intersects $B$ non-empty. Therefore, $B=\bigcup_{j=0}^{i}B_j$.\\
\indent By construction, we know that if $\gamma_e\neq \gamma_f$, then $B_e\neq B_f$ with $0\le e,f\le i$. Moreover, each pair of sets in the above collection is disjoint. That is, if $B_e\neq B_f$, then $B_e\cap B_f=\emptyset$ with $0\le e,f\le i$ and $e\neq f$. Otherwise, there exists an element $\gamma\in B_e\cap B_f$. Let $B_e=B\cap \mathcal{G}_{\gamma_e,\delta}$ and $B_f=B\cap \mathcal{G}_{\gamma_f,\epsilon}$ with $\delta,\epsilon\in \{1,2\}$. If $\delta=\epsilon$, then $\mathcal{G}_{\gamma_e,\delta}=\mathcal{G}_{\gamma_f,\delta}$, and so $B_e=B_f$, a contradiction. If $\delta\neq \epsilon$, then $\mathcal{G}_{\gamma,\delta}=\mathcal{G}_{\gamma_f,\delta}$ and $\mathcal{G}_{\gamma,\epsilon}=\mathcal{G}_{\gamma_f,\epsilon}$, contrary to $B$ intersects only one Cartesian line through $\gamma$.\\
\indent Thus, $|B_j|\mid |B|$, i.e., $(x_1+1)\mid k$. This completes the proof of part (1).\\
\indent (2) Let $(\alpha,B)$ be the flag such that $B$ intersects two Cartesian line through $\alpha=(\alpha_{1},\alpha_{2})$. Let $r_1=|(B\cap \mathcal{G}_{\alpha,1})\setminus \{\alpha\}|$ and $r_2=|(B\cap \mathcal{G}_{\alpha,2})\setminus \{\alpha\}|$. It is easy to know that $r_1, r_2\ge 1$ and $x_1=r_1+r_2$.\\
\indent We claim that for any blocks $C$ through $\alpha$, then $C$ intersect the two Cartesian lines through $\alpha$. In fact, by the flag-transitivity of $G$, there exists $g=(g_1,g_2;\sigma) \in G$ such that $(\alpha,B)^g=(\alpha,C)$. Similar to Part (1), if $g=(g_1,g_2;(1))$, then $\alpha_{1}^{g_1}=\alpha_{1}$ and $\alpha_{2}^{g_2}=\alpha_{2}$. It follows that $\mathcal{G}_{\alpha,1}^g=\mathcal{G}_{\alpha,1}$ and $\mathcal{G}_{\alpha,2}^g=\mathcal{G}_{\alpha,2}$. So 
\begin{equation*}
	\begin{split}
		|C\cap  \mathcal{G}_{\alpha,1}|=|B\cap  \mathcal{G}_{\alpha,1}|=r_1+1,\\ |C\cap  \mathcal{G}_{\alpha,2}|=|B\cap  \mathcal{G}_{\alpha,2}|=r_2+1.
	\end{split}
\end{equation*}
 If $g=(g_1,g_2;(12))$, then $\mathcal{G}_{\alpha,1}^g=\mathcal{G}_{\alpha,2}$ and $\mathcal{G}_{\alpha,2}^g=\mathcal{G}_{\alpha,1}$. It follows that $|C\cap  \mathcal{G}_{\alpha,2}|=|B\cap  \mathcal{G}_{\alpha,1}|=r_1+1$ and $|C\cap  \mathcal{G}_{\alpha,1}|=|B\cap  \mathcal{G}_{\alpha,2}|=r_2+1.$ The claim is proved.\\
\indent In the same way as in Part (1), the conclusion that the block $B$ intersects the two Cartesian line through $\beta$ holds for any points $\beta\in B$. Furthermore, $|B\cap \mathcal{G}_{\beta}|=x_1$. Note that there are $k$ blocks containing $\alpha$. Assume that there are $k_1$ blocks intersecting with $\mathcal{G}_{\alpha,1}\setminus \{\alpha\}$ as $r_1$, and $k_2$ blocks intersecting it as $r_2$. It follows that $k=k_1+k_2$. Counting $\Sigma_2=\{(\beta,B)|\beta \in B,\beta\neq \alpha,\alpha\in B, \beta\in \mathcal{G}_{\alpha,1}\}$ in two way, we conclude that $\lambda(w-1)=k_1r_1+k_2r_2$. But from Lemma \ref{lemma3.2}(1) the equation $2\lambda(w-1)=kx_1=(k_1+k_2)(r_1+r_2)$ also holds, then $(k_1-k_2)(r_1-r_2)=0$, and so $k_1=k_2=\frac{k}{2}$ or $r_1=r_2=\frac{x_1}{2}$.\\
\indent We assert that if $k_1=k_2=\frac{k}{2}$, then $r_1=r_2=\frac{x_1}{2}$. Assume that $k_1=k_2=\frac{k}{2}$, which means that $k$ is even. Note that $k=k^*p$ and $p$ is an odd prime, then $k^*$ must be even. Since $k^*\mid (v-1)$, then $v$ is odd. Recall that $x_1=\frac{2\lambda(w-1)}{k}$, and since $\lambda (w^2-1)=k(k-1)$, then $x_1=\frac{2(k-1)}{w+1}$. It implies that $k-1=x_1 \cdot \frac{w+1}{2}$, and so $x_1$ is odd. By Lemma \ref{lemma2.4}, we know that $|\alpha^{M_B}|=\frac{k}{2}$ and $B_{20}\cap B_{21}=\emptyset$. It follows that $|\alpha^{M_B}|=|B\cap \mathcal{G}_{\alpha,2}||B_{20}|$. Choose the two blocks $B, C$ through $\alpha$ such that $|(B\cap \mathcal{G}_{\alpha,2})\setminus \{\alpha\}|=r_1$ and $|(C\cap \mathcal{G}_{\alpha,2})\setminus \{\alpha\}|=r_2$. Recall that $B_{20}=\{\alpha_{1}|(\alpha_{1},\alpha_{2})\in \alpha^{M_{B}}\}$ and $C_{20}=\{\alpha_{1}|(\alpha_{1},\alpha_{2})\in \alpha^{M_{C}}\}$. Since $G$ is flag-transitive, then $|B_{20}|=|C_{20}|$, and hence $|B\cap \mathcal{G}_{\alpha,2}|=|C\cap \mathcal{G}_{\alpha,2}|$. This means that $r_1=r_2=\frac{x_1}{2}$.\\
\indent Putting $B_0=B\cap \mathcal{G}_{\alpha,2}$. The number of elements in $B_0$ is $\frac{x_1}{2}+1$. Taking an element $\gamma_1\in B\setminus B_0$, we know that the number of $B$ intersects the second Cartesian line through $\gamma_1$ is $\frac{x_1}{2}+1$. Let $B_1=B\cap \mathcal{G}_{\gamma_1,2}$. We repeat the above steps until there are no more elements in $B$, and we obtain some points $\gamma_0=\alpha$, $\gamma_1,\cdots,\gamma_i$, where $\gamma_j\in B\setminus (B_0\cup\dots \cup B_{j-1})$ and $B_j=B\cap \mathcal{G}_{\gamma_j,2}$ for $0\le j\le i$. Thus, $B=\bigcup_{j=0}^{i}B_j$, where $|B_j|=\frac{x_1}{2}+1$ for $0\le j\le i$. Obviously, if $\gamma_e\neq \gamma_f$, then $B_e\neq B_f$ with $0\le e,f\le i$. Moreover, every two sets in the above collection are disjoint. If there exists an element $\gamma\in B_e\cap B_f$, where $B_e=B\cap \mathcal{G}_{\gamma_e,2}$ and $B_f=B\cap \mathcal{G}_{\gamma_f,2}$ for $0\le e\neq f\le i$. It implies that $\mathcal{G}_{\gamma_e,2}=\mathcal{G}_{\gamma_f,2}$, so that $B_e=B_f$, a contradiction. Therefore, $|B_j|\mid |B|$, and so $(\frac{x_1}{2}+1)\mid k$.$\hfill\square$\vspace{3mm}
\begin{lemma}\label{prop2.7} {\rm (\cite[Lemma 3.12]{2023})}
	Let $\mathcal{D}$ be a symmetric $2$-$(v,k,\lambda)$ design, admitting a flag-transitive, point-primitive group $G\le Aut(\mathcal{D})$. Assume that $v = w^2$ and $\gcd(k,\lambda)=p$, where $p$ is a prime. If for any point $\alpha$ in the design, each block that contains it intersects the two Cartesian lines
	through $\alpha$, then either $G$ does not have a product action or $\mathcal{D}$ is a Menon design.
\end{lemma}

Under the conditions of Lemma \ref{prop2.7}, we know that if the automorphism group $G$ is of product type, then $\mathcal{D}$ is a Menon design. It is well known that the existence of a symmetric design with parameters
$(4n^2,2n^2-n,n^2-n)$ is equivalent to the existence of a regular Hadamard
matrix of order $4n^2$. Such symmetric designs and its complementary designs
are called \emph{Menon designs} of order $n^2$. Note that $\gcd(k,\lambda)=n=p$ is a prime number, which forces $k-\lambda=p^2$. By \cite{lu}, we know $G$ is not of product type, a contradiction. Combining it with Lemma \ref{prop2.6}, we obtain the following corollary:
\begin{corollary}\label{coro1}
	If a symmetric $2$-$(v,k,\lambda)$ design $\mathcal{D}$ with $\gcd(k,\lambda)$ is an odd prime, admitting a flag-transitive point-primitive automorphism group $G$. Assume that $G$ is of product action type with $v=w^2$, then $B$ intersects only one Cartesian line through $\alpha$, where $(\alpha,B)$ is a flag of $\mathcal{D}$.
\end{corollary}

\indent Now $G \leq H\wr S_{2}=(H\times H)\rtimes S_{2}$ has a product action on $\mathcal{P}=\Gamma^2,$ where $H$ is a primitive group of almost simple or simple diagonal type acting on a set $\Gamma$. By Lemma \ref{lemma3.2}(1) we get $k\mid 2\lambda(w-1)$, and so $k^*\mid 2(w-1)$. Then there exists an integer $x$ such that $2(w-1)=k^*x$. It implies that 
\begin{equation*}
	x=\dfrac{2(w-1)}{k^*}.
\end{equation*}
Recall that $x_{1}=\frac{2\lambda(w-1)}{k}$, then $x_1=\lambda^*x.$\\
\indent By Corollary \ref{coro1}, Lemma \ref{prop2.6}(1) holds. Namely, there is a point $\alpha$ and a block $B$ through $\alpha$ such that $B$ intersects only one Cartesian line through $\alpha$. So we assume that $(x_1+1)\mid k$ in the following.
\begin{lemma}\label{l3.3}
 Suppose that $\mathcal{D}$ is a symmetric $2$-design with $\gcd(k,\lambda)$ prime, and $G\le Aut(\mathcal{D})$ is a flag-transitive point-primitive product action group. In addition, assume that $G\le H\wr S_{2}$ acts on $\mathcal{P}=\Gamma^{2}$. If $(x_{1}+1)\mid k$, then $k^*=4(x_{1}+1)$.
\end{lemma}
{\bfseries Proof}. By assumption, there exists an integer $a$ such that $k=a(x_1+1)$. On the other hand, $x_1=\frac{2\lambda(w-1)}{k}$, and since $\lambda (w^2-1)=k(k-1)$, then $x_1=\frac{2(k-1)}{w+1}$, which implies that $k=\frac{x_1(w+1)}{2}+1$. So, $a(x_1+1)=\frac{x_1(w+1)}{2}+1$. It follows that $w-1=(x_1+1)(w-2a+1)$. Hence,
\begin{equation}\label{eq3.55}
	k^*x=2(w-1)=2(x_1+1)(w-2a+1),
\end{equation}
 and $(x_1+1)\mid k^*x$. So that $(x_1+1)\mid k^*$ as $\gcd(x_1+1,x)=\gcd(\lambda^*x+1,x)=1$.\\
\indent Since $k^*$ is a divisor of $k=\frac{x_1(w+1)}{2}+1$, then $k^*\mid (2x_1(w+1)+4)$. Moreover, $k^*$ also divides $2(w-1)$, then $k^*\mid \gcd(2x_1(w+1)+4, 2(w-1)),$ and so $k^*\mid 4(x_1+1)$. Therefore, $k^*=x_{1}+1$, $2(x_1+1)$ or $4(x_1+1)$.\\
\indent Note that $\ell=2$, then 
\begin{equation}\label{eqqq3.7}
	4p=\lambda^*x^2+\dfrac{2x(\lambda^*x+1)}{w-1}.
\end{equation}
by Lemma \ref{lemma3.2}(2).\vspace{2mm}\\
{\bfseries Claim 1}. $k^*\neq 2(x_1+1)$.\\
\indent If $k^*=2(x_{1}+1)$, then
$k^*x=2x(x_1+1)=2x(\lambda^*x+1)$, and so 
$2x(\lambda^*x+1)=2(w-1)$ by $k^*x=2(w-1)$. Combining it with (\ref{eqqq3.7}), then $4p=\lambda^* x^2+2$, i.e., 
\begin{equation}\label{eqq3.6}
xx_1=2(2p-1).
\end{equation}
\indent Now, by (\ref{eqq3.6}), we know $2\mid x$ or $2\mid x_1$. Note that $k^*=2(x_{1}+1)$, then $2\nmid \lambda^*$ since $\gcd(k^*,\lambda^*)=1$. If $2\mid x_1$, then from $x_1=\lambda^*x$ we have $2\mid x$. Therefore $4\mid xx_1$ which contradict to Equation (\ref{eqq3.6}). If $2\mid x$, then $2\mid x_1$ by $x_1=\lambda^*x$. This is also impossible. \vspace{2mm}\\
{\bfseries Claim 2}. $k^*\neq x_1+1$.\\
\indent Suppose that $k^*=x_1+1$, then $k^*x=x(x_1+1)=x(\lambda^*x+1)$, and hence $2x(\lambda^*x+1)=4(w-1)$. By (\ref{eqqq3.7}), we have 
\begin{equation}\label{eq3.6}
	\lambda^*x^2=xx_1=x(k^*-1)=4(p-1).
\end{equation}
Equation (\ref{eq3.55}) gives that  $x=\frac{2(x_1+1)(w-2a+1)}{k^*}=2(w-2a+1)$, and so $x$ is even.\vspace{1mm}\\
\indent We derive the contradiction by (1)-(4) in the following.\\
\indent (1) $x\ge 4$ and $p\ge 11$. In fact, if $x=2$, then $w=2a$. Note that $k=k^*p=p(x_1+1)=a(x_1+1)$, which means that $a=p$. So $w=2p$, $k^*=w-1=2p-1$ and $\lambda^*=\frac{4(p-1)}{x^2}=p-1$. This implies that $k=p(2p-1)$, $\lambda=p(p-1)$ and $k-\lambda=(2p^2-p)-(p^2-p)=p^2$. The parameters can be excluded by \cite{lu}. Moreover, Equation (\ref{eq3.6}) yields $4(p-1)\ge 16\lambda^*\ge 32,$ and thus $p\ge 11$.\vspace{1mm}\\
\indent (2) $\gcd(p,w)=1$ and $\gcd(p,|Out(T)|)=1$, where $T$ is a transitive non-abelian simple group of degree $w$. Using Lemma \ref{l1}(1), we have $\lambda v<k^2$, and so $w^2<(\lambda^*x^2+2x+1)p$. It implies that $4wp+p>w^2+4p^2$. Assume that $p\mid w$, then $w=2p$. This is impossible since $x=2(w-2p+1)=2$.\\
\indent Clearly, if $T$ is a sporadic simple group or an alternating group, then $\gcd(p,|Out(T)|)=1$ since $|Out(T)|=1,2,4$ and $p\ge 11.$ Moreover, if $T=L_2(9)$, then $\gcd(p,|Out(L_2(9))|)=1$. Now, we consider $T$ is the remaining simple group of Lie type over the finite field $\mathbb{F}_{q}$ and $q=p_{1}^{f}$, where $p_{1}$ is a prime. Assume that $p\mid |{Out}(T)|$. Combining it with $p\ge 11$, we have $p \mid f$ when $T\neq U_{n}(q)$ and $L_{n}(q)$. Recall that $x\ge 4,$ then $p\ge k^*$ by Equation (\ref{eq3.6}). So 
\begin{equation*}
	w^{2}=v<k^{2}\le p^{4}\le f^{4}< q^{2},
\end{equation*}
and hence $w<q$. Note that $T$ is a transitive group of degree $w$, then the minimal degree of $T$ is not greater than $w$, and then the minimal degree is less that $q$. This is a contradiction by \cite[Lemma 2.1]{2020}. When 
$T=U_{n}(q)$ or $L_{n}(q)$, the prime number $p$ also probably divides $\gcd(n,q\pm 1)$. It follows that $11\le p\le n$ and $p\le q\pm 1$. Then $w<p^{2} \leq(q+1)^{2}$, and so the minimal degree of $T$ is less than $(q+1)^2$, which is also a contradiction by \cite[Lemma 2.1]{2020}.\vspace{1mm}\\
\indent (3) The primitive group $H$ is of almost simple type such that $G\le H\wr S_2$. Assume that $H$ is not of almost simple type, and so it is a simple diagonal group by the definition of product action type, putting $Soc(H)=T^{m_{1}}$, where $T$ is a non-abelian simple group and $m_{1}\ge 2$. It follows that $w=|T|^{m_1-1}$, and so $v=w^2=|T|^{2(m_1-1)}.$ Note that the point-stabilizer $G_{\alpha}\lesssim (Aut(T) \times S_{m_{1}})\wr S_{2}$, where $\alpha\in \mathcal{P}$. Lemma \ref{l1}(2) yields $k\mid  2m_{1}!^2|Aut(T)|^2,$ and hence $p\mid m_1!\cdot |Aut(T)|$. If $p\mid m_1!$, then $k\le p^2\le m_{1}^2$, and so $|T|^{2(m_{1}-1)}=v<k^2\le m_{1}^4$. Since $T$ is a non-abelian simple group, then $60^{2(m_{1}-1)}\le |T|^{2(m_{1}-1)}< m_{1}^4.$ It means that $m_{1}=1$, contradicting the fact $m_{1}\ge 2$. Hence, $p\mid |Aut(T)|$, by \cite[Lemma 2.4]{2024}, then $p^2\le |T|$. It follows that $|T|^{2(m_{1}-1)}=v<k^2\le p^4\le |T|^2$, and so $m_1=1$, a contradiction.\vspace{1mm}\\
\indent (4) The final contradiction. Now $T_{1}\times T_{2} \unlhd G\le H\wr S_{2}$, where $T_i\cong T(i=1,2)$ is a non-abelian simple group, $H$ is a primitive almost simple group. Let $\alpha \in B$ and $B$ be a block. From Lemma \ref{lemma3.3}(1), we obtain that $k\mid 2|\alpha^{\left(T_{1}\right)_{B}}||\alpha^{\left(T_{2}\right)_{B}}||Out(T)|^2w^4$, and thus $p\mid 2|\alpha^{\left(T_{1}\right)_{B}}||\alpha^{\left(T_{2}\right)_{B}}||Out(T)|^2w^4$. And since $\gcd(p,|Out(T)|^2w^4)=1$, it implies that $p\mid|\alpha^{\left(T_{1}\right)_{B}}||\alpha^{\left(T_{2}\right)_{B}}|$, and hence $p\mid |\alpha^{\left(T_{1}\right)_{B}}|$ or $p\mid |\alpha^{\left(T_{2}\right)_{B}}|$. Noting that $\alpha^{(T_{i})_{B}}\subseteq \alpha^{T_{i}}\cap B$ for $i=1,2$, then $|\alpha^{T_{1}}\cap B|+|\alpha^{T_{2}}\cap B|-2\ge p-1.$ Further, Lemma \ref{lemma3.3}(2) gives that $|\alpha^{T_{1}}\cap B|+|\alpha^{T_{2}}\cap B|-2=\frac{2(k-1)}{w+1}$, and so $\frac{2(k-1)}{w+1}\ge p-1$. Therefore,  
\begin{equation*}
	\dfrac{2p(w-1)}{x}=k\ge \dfrac{(p-1)(w+1)}{2}+1.
\end{equation*}
It follows that $x\le 4$, and hence $x=4$. Therefore, $w=2p+1$, $k^*=p$ and $k=p^2$. This means that $|\alpha^{T_{1}}\cap B|+|\alpha^{T_{2}}\cap B|-2=p-1$, that is,
\begin{equation}\label{eq3.7}
	|\alpha^{T_{1}}\cap B|+|\alpha^{T_{2}}\cap B|=p+1.
\end{equation} 
For each $i=1,2$, the size of $\alpha^{T_i}\cap B$ are factors of $k$, where $|\alpha^{T_i}\cap B|\le w$ and $\frac{k}{|\alpha^{T_i}\cap B|}\le w$. By Equation (\ref{eq3.7}), without loss of generality, let $|\alpha^{T_{1}}\cap B|=1$ and $|\alpha^{T_{2}}\cap B|=p$. It implies that $\frac{k}{|\alpha^{T_1}\cap B|}=p^2$, which is contrary to $\frac{k}{|\alpha^{T_1}\cap B|}\le w=2p+1$. Thus, $k^*\neq x_1+1$.\\
\indent From Claims 1 and 2, we conclude that $k^*=4(x_1+1).$
$\hfill\square$\vspace{3mm}
\begin{proposition}\label{prop3.4}
	If $\mathcal{D}=(\mathcal{P}, \mathcal{B})$ is a symmetric $2$-$(v, k, \lambda)$ design admitting a flag-transitive point-primitive automorphism group $G$, where  $\gcd(k,\lambda)=p$ is a prime, then $G$ is not of product action type.
\end{proposition}
{\bfseries Proof}. Suppose that $G$ is of product action type, by Lemma \ref{lp}, we have $G\le H\wr S_2$, where $H$ is of almost simple or diagonal type on the $\Gamma$ of size $w$. Lemma \ref{l3.3} gives that $k^*=4(x_1+1)$, we get $k^*x=4x(x_1+1)=2(w-1)$, and so $w-1=2x(x_1+1)=2x(\lambda^*x+1).$ It follows from Equation (\ref{eqqq3.7}) that
\begin{equation}\label{eqqq3.77}
	xx_1=\lambda^*x^2=4p-1.
\end{equation}
 It is easy to see that $\lambda^*$ and $x$ are both odd numbers.\\
  \indent The equation $\lambda^*(v-1)=k^*(k-1)$ gives that $k^*\mid (v-1)$, so that $v$ is odd. From Lemma \ref{lemma2.4} and the previous statements, we obtain that $(|B^M|,|\alpha^{M_B}|)=(v,\frac{k}{2})$. Note that $|\alpha^{M_B}|=\frac{k}{2}$, that is, $M_{B}$ is not transitive on $B$. Recall that $c_1=\left|\alpha^{T_1} \cap \alpha^{M_B}\right|$ and  $c_2=\left|\alpha^{T_2} \cap \alpha^{M_B}\right|.$
   Using Lemma \ref{l2}, we get $$x_1=c_1+c_2-2.$$
   In the following, without loss of generality, we assume that $c_1<c_2$ since $x_1$ is odd, and thus $c_1\le \frac{x_1+1}{2}$.
    Moreover, $B_{10}=\{\alpha_{2}|(\alpha_{1},\alpha_{2})\in \alpha^{M_{B}}\}$,
$ B_{20}=\{\alpha_{1}|(\alpha_{1},\alpha_{2})\in \alpha^{M_{B}}\}.$ Then  $\frac{k}{2}=c_1|B_{10}|=c_2|B_{20}|$, this implies that $c_1|B_{10}||B_{20}|=\frac{k}{2}|B_{20}|$ and $c_2|B_{10}||B_{20}|=\frac{k}{2}|B_{10}|$. Adding the two previous equations, we conclude that $|B_{10}||B_{20}|(c_1+c_2)=\frac{k}{2}(|B_{10}|+|B_{20}|)$, and so \begin{equation}\label{eq3.9}
 	|B_{10}||B_{20}|(x_1+2)=2p(x_1+1)(|B_{10}|+|B_{20}|).
 \end{equation}
 \indent Obviously, $p>x$. In fact, if $p\le x$, then $4p-1=\lambda^*x^2\ge 2p^2$, which contradicts $p\ge 3$. Moreover, since $k=k^*p$, then $\frac{x_1(w+1)}{2}+1=4p(x_1+1)$, i.e.,  $(\frac{w+1}{2}-4p)x_1=4p-1$. Note that $x_1=\lambda^*x$ is odd and $\lambda^*\ge 2$, then $4p-1\ge 3(\frac{w+1}{2}-4p)$, and so $32p-5\ge 3w$. It follows that $w\le 10p$. \vspace{2mm}\\
 {\bfseries Claim}. $\gcd(p,x_1+2)=1$ and $\gcd(p,|B_{20}|)=1$.\\
 \indent If $p\mid (x_1+2)$, then $x_1+2=\lambda^*x+2=pt_1$ for some integer $t_1$. By (\ref{eqqq3.77}), we have 
 \begin{equation}\label{eq3.10}
 	xt_1p+1=2x+4p.
 \end{equation}
 Now $p>x$ forces $xt_1\le 5.$ On the other hand, by (\ref{eq3.10}), $(xt_1-4)p=2x-1$, this implies that $xt_1\ge 5$. So, $xt_1=5$, and thus $x=1$ or $5$. These are impossible as $p=2x-1$ is not a prime.\\
 \indent Let $|B_{20}|=ps$ for some integer $s$. Recall that $c_1\le \frac{x_1+1}{2}$, then $c_2\ge \frac{x_1+3}{2}>\frac{x_1+1}{2}$. So, 
 \begin{equation*}
 	2(x_1+1)=sc_2>s \left(\frac{x_1+1}{2}\right),
 \end{equation*}
 and hence $s\le 3$.\\
 \indent If $s=1$, then $c_2=2(x_1+1)$, contrary to $c_1+c_2=x_1+2>c_2.$ If $s=2$, then $c_2=x_1+1$ and $c_1=1$. This implies that $|B_{10}|=\frac{k}{2}=\frac{x_1(w+1)+2}{4}\le w$, and so $x_1\le 3.$ On the other hand, $x_1=\lambda^*x\ge 2x$ and $x_1$ is odd, then $x_1=3$. It follows that $x=1$, contradicts Equation (\ref{eqqq3.77}). Therefore, $s=3$ and $c_2=\frac{2(x_1+1)}{3}$, $c_1=\frac{x_1+4}{3}$. Note that $\frac{k}{2}=2p(x_1+1)=c_1|B_{10}|$, then $p\mid c_1|B_{10}|$. If $p\mid c_1$, then $c_1=pz$ for some integer $z$, and thus $pz=\frac{x_1+4}{3}$. Combining it with (\ref{eqqq3.77}), we have
 \begin{equation}\label{eq3.11}
 	3xzp+1=4x+4p.
 \end{equation} 
Recall that $p>x$, then $xz\le 2$, and so $x=z=1$ by $xz$ is odd. But they do not satisfy (\ref{eq3.11}). Hence, $p\mid |B_{10}|$. Let $|B_{10}|=pt$ for some integer $t.$ Taking it and $c_1=\frac{x_1+4}{3}$ into $2p(x_1+1)=c_1|B_{10}|$, we get 
 $t(x_1+4)=6(x_1+1)$. 
 It is easy to know that $(t,x_1)=(3,2), (4,5)$ or $(5,14)$. Since $x_1$ is odd, then $x_1=5$. Recall that $x_1=\lambda^*x$ and $\lambda^*\neq 1$, then $x=1$ and $\lambda^*=5$. However, no positive integer $p$ satisfying (\ref{eqqq3.77}). Hence the claim is proved.\vspace{1mm}\\
 \indent Equation (\ref{eq3.9}) yields $p\mid |B_{10}||B_{20}|(x_1+2)$. Combined it with Claim, we obtain that $p\mid |B_{10}|$. Assume that $|B_{10}|=pt$ for some integer $t$. From $\frac{k}{2}=2p(x_1+1)=tpc_1$, we get $2(x_1+1)=tc_1$. Recall that $c_1\le \frac{x_1+1}{2}$, we know $t=\frac{2(x_1+1)}{c_1}\ge 4.$ On the other hand, $|B_{10}|=pt\le w\le 10p$, so $t\le 10$. Hence, 
 \begin{equation}\label{eqqq3.9}
 	c_1=\dfrac{2(x_1+1)}{t} ~\text{and}~ c_2=\dfrac{x_1(t-2)+2(t-1)}{t},
 \end{equation}
 where $4\le t\le 10.$\\
 \indent For $4\le t\le 10$, there is a positive integer $y$ such that $c_2=py$ as $p\mid c_2|B_{20}|$ and $\gcd(p,|B_{20}|)=1$. It follows from (\ref{eqqq3.9}) that $py=\frac{x_1(t-2)+2(t-1)}{t}$. Combining it with (\ref{eqqq3.77}), we conclude that 
 \begin{equation}\label{eqqq3.10}
 	xytp+(t-2)=(2t-2)x+(4t-8)p.
 \end{equation}
Recall that $p>x$, it follows that $(6-xy)t>10$. And since $4\le t\le 10,$ then $xy\le 4$. Note that $x$ is odd, we obtain all possibilities of $(t,x,y,p)$ that satisfying Equation (\ref{eqqq3.10}), i.e., $(t,x,y,p)=(5,3,1,7), (6,1,3,3),$ $(7,1,3,7)$, $(6,3,1,13)$ or $(7,3,1,31).$ The last two cases are ruled out since $\lambda^*=\frac{4p-1}{x^2}$ is an integer by (\ref{eqqq3.77}). For the remaining three possibilities, we obtain the parameters $(v,k,\lambda)=(61^2,280,21)$, $(25^2,144,33)$ or $(57^2,784,189)$, respectively.\\
\indent Now we rule out these three possibilities with the help of the computer algebra system {\sc Magma} \cite{magma}. If $(v,k,\lambda)=(61^2,280,21)$, we search for the possible candidates for the group $H$ by the command {\tt H:=PrimitiveGroups(w)} where $w=61$. It follows that the primitive group $H$ is $A_{61}$ or $S_{61}$, and so $G=H\wr S_2$. Then for these cases $G=A_{61}\wr S_2$ or $S_{61}\wr S_2$, we find the subgroups $G_1$ with index $v$ which are the candidates for the block-stabilizer $G_{B}$. For each candidates $G_1$ for $G_{B}$, by the command {\tt Orbits($G_1$)}, we get the lengths of the 
$G_{B}$-orbits acting on $\mathcal{P}$ are  $1, 120$ and $3600$. This contradicts the flag-transitivity of $G$ due to $G_B$ has at least one orbit with size $k=280$. Similarly, if $(v,k,\lambda)= (25^2,144,33)$, then the lengths of the $G_{B}$-orbits are $1, 48$ and $576$, a contradiction. For the case $(v,k,\lambda)=(57^2,784,189)$, it follows that $H=PSL_2(19)$, $PSL_3(7)$, $PSL_3(7).3$, $A_{57}$ or $S_{57}$. Clearly, $H\neq A_{57}, S_{57}$. For the remaining groups $H$, let $K=H\wr S_2$, we can determine the orbit lengths of $K_B$. Note that $K$ is flag-transitive, then there exists at least one orbit $O$ of $K_B$ with size $k=784$, and so  
$K=PSL_3(7)\wr S_2$ or $(PSL_3(7).3)\wr S_2$. But for each $K$, the structure $(\mathcal{P},O^{K})$ is not a $2$-design. Since $G\le K\le Sym(\mathcal{P})$, by Lemma \ref{ll2.2}, we know that there is no $2$-$(57^2,784,189)$ design admitting a flag-transitive automorphism group $G$.\\
  \indent In short, by analyzing above, $G$ is not of product action.$\hfill\square$\vspace{3mm}\\
{\bf Proof of Theorem \ref{TH1}} It follows immediately from Propositions \ref{prop3.11}, \ref{prop3.1} and \ref{prop3.4}.
$\hfill\square$\vspace{3mm}\\
	{\bf {Declaration of competing interest}}\\
	\indent The authors declare that there is no conflict of interest in this paper.\vspace{3mm}\\
	{\bf {Data availability}}\\
	\indent Data sharing not applicable to this article as no datasets were generated or analyzed during the current study.\vspace{3mm}\\
	{\bf {Acknowledgements}}\\
	\indent The authors are grateful to Professor A. Montinaro for valuble suggestion on the proof of Lemma \ref{prop2.6}.


\begin{thebibliography}{99}
	\bibitem{magma}	W. Bosma, J. Cannon, C. Playoust, The Magma algebra system. I. The user language, J. Symbolic Comput., 24(3-4)(1997), 235-265.
		
		
	\bibitem{ca}P. J. Cameron, C. E. Praeger,
	Block-transitive $t$-designs. I. point-imprimitive designs, Discrete Math., 118(1-3)(1993), 33-43.
	
	\bibitem{chen}J. F. Chen, J. X. Shen, S. L. Zhou, Reduction for flag-transitive symmetric designs with $k>\lambda(\lambda-2)$, Discrete Math., 346(10)(2023), 113557.
		
    \bibitem{Dav1} H. Davies,  Flag-transitivity and primitivity, Discrete Math., 63(1)(1987), 91-93.
		
	\bibitem{Dembowski} P. Dembowski, Finite Geometries, Springer-Verlag, New York, 1968.			
		
	\bibitem{Osifodunrin}E. S. Lander, Symmetric designs: an algebraic approach, London Mathematical Society Lecture Note Series, vol. 74., Cambridge University Press, Cambridge, 1983.
	
	\bibitem{li}M. W. Liebeck, C. E. Praeger, J. Saxl, On the O'Nan-Scott theorem for finite primitive permutation groups, J. Austral. Math. Soc. Ser. A,  44(3)(1988), 389-396.
	
	\bibitem{lu}Z. W. Lu, S. L. Zhou, Flag-transitive point-primitive symmetric 2-designs of prime square order,
	J. Algebra Appl., 25(8)(2026), 2650069.
	
	
    \bibitem{regueiro}E. O'Reilly-Regueiro, On primitivity and reduction for flag-transitive symmetric designs, J. Combin. Theory Ser. A, 109(1)(2005), 135-148.
	
	\bibitem{2023}E. O'Reilly-Regueiro, J. E.  Rodr{\'i}guez-Fitta, A parametrisation for symmetric designs admitting a flag-transitive, point-primitive automorphism group with a product action, Ars Math. Contemp., 23(2)(2023), 4-18.
		
	\bibitem{100}L. Tang, J. Chen, Z. L. Zhang,
	Classification of flag-transitive point-primitive non-symmetric $2$-$(v,k,\lambda)$ designs with $v<100$.
	Appl. Math. Comput., 430(2022), 127278.
		
		
	\bibitem{zhao} Y. W. Zhao, S. L. Zhou, Reduction for flag-transitive $2$-$(v,k,\lambda )$  designs with $(r,\lambda)=2$, J. Combin. Des., 29(9)(2021), 629-643.
		
	\bibitem{2020}Y. L. Zhang, Z. L. Zhang, S. L. Zhou, Reduction for primitive flag-transitive symmetric $2$-$(v,k,\lambda)$ designs with $\lambda$ prime, Discrete Math., 343(6)(2020), 111843.
	
	\bibitem{2024}Y. L. Zhang, J. F. Chen, 
	Reduction for flag-transitive point-primitive $2$-$(v,k,\lambda)$ designs with $\lambda$ prime,
	J. Combin. Des., 32(2)(2024), 88-101.
	
	
	\bibitem{on} C. Y. Zhong, S. L. Zhou, On flag-transitive automorphism groups of $2$-designs, Discrete Math., 346(2)(2023), 113227.
		
	\bibitem{1988}P. H. Zieschang, Flag transitive automorphism groups of $2$-designs with $(r,\lambda)=1$,
J. Algebra, 118(2)(1988), 369-375.		
	\end{thebibliography}
\end{document}